\documentclass[12pt,reqno]{amsart}
\usepackage{amscd,amssymb,amsmath,mathabx,etex,tikz,tikz-cd,float}
\newtheorem{thm}[equation]{Theorem}
\numberwithin{equation}{section}

\newtheorem{rmk}[equation]{Remark}

\newtheorem{lem}[equation]{Lemma}

\newtheorem{defin}[equation]{Definition}

\newtheorem{prop}[equation]{Proposition}

\begin{document}
\raggedbottom \voffset=-.7truein \hoffset=0truein \vsize=8truein
\hsize=6truein \textheight=8truein \textwidth=6truein
\baselineskip=18truept

\def\mapright#1{\ \smash{\mathop{\longrightarrow}\limits^{#1}}\ }
\def\mapleft#1{\smash{\mathop{\longleftarrow}\limits^{#1}}}
\def\mapup#1{\Big\uparrow\rlap{$\vcenter {\hbox {$#1$}}$}}
\def\mapdown#1{\Big\downarrow\rlap{$\vcenter {\hbox {$\ssize{#1}$}}$}}
\def\mapne#1{\nearrow\rlap{$\vcenter {\hbox {$#1$}}$}}
\def\mapse#1{\searrow\rlap{$\vcenter {\hbox {$\ssize{#1}$}}$}}
\def\mapr#1{\smash{\mathop{\rightarrow}\limits^{#1}}}
\def\ss{\smallskip}
\def\ar{\arrow}
\def\vp{v_1^{-1}\pi}
\def\at{{\widetilde\alpha}}
\def\sm{\wedge}
\def\la{\langle}
\def\ra{\rangle}
\def\on{\operatorname}
\def\ol#1{\overline{#1}{}}
\def\spin{\on{Spin}}
\def\cat{\on{cat}}
\def\lbar{\ell}
\def\qed{\quad\rule{8pt}{8pt}\bigskip}
\def\ssize{\scriptstyle}
\def\a{\alpha}
\def\bz{{\Bbb Z}}
\def\Rhat{\hat{R}}
\def\im{\on{im}}
\def\ct{\widetilde{C}}
\def\ext{\on{Ext}}
\def\sq{\on{Sq}}
\def\eps{\epsilon}
\def\ar#1{\stackrel {#1}{\rightarrow}}
\def\br{{\bold R}}
\def\bC{{\bold C}}
\def\bA{{\bold A}}
\def\bB{{\bold B}}
\def\bD{{\bold D}}
\def\bh{{\bold H}}
\def\bQ{{\bold Q}}
\def\bP{{\bold P}}
\def\bx{{\bold x}}
\def\bo{{\bold{bo}}}
\def\si{\sigma}
\def\Vbar{{\overline V}}
\def\dbar{{\overline d}}
\def\wbar{{\overline w}}
\def\Sum{\sum}
\def\tfrac{\textstyle\frac}
\def\tb{\textstyle\binom}
\def\Si{\Sigma}
\def\w{\wedge}
\def\equ{\begin{equation}}
\def\AF{\operatorname{AF}}
\def\b{\beta}
\def\G{\Gamma}
\def\D{\Delta}
\def\L{\Lambda}
\def\g{\gamma}
\def\k{\kappa}
\def\psit{\widetilde{\Psi}}
\def\tht{\widetilde{\Theta}}
\def\psiu{{\underline{\Psi}}}
\def\thu{{\underline{\Theta}}}
\def\aee{A_{\text{ee}}}
\def\aeo{A_{\text{eo}}}
\def\aoo{A_{\text{oo}}}
\def\aoe{A_{\text{oe}}}
\def\vbar{{\overline v}}
\def\endeq{\end{equation}}
\def\sn{S^{2n+1}}
\def\zp{\bold Z_p}
\def\cR{{\mathcal R}}
\def\P{{\mathcal P}}
\def\cF{{\mathcal F}}
\def\cQ{{\mathcal Q}}
\def\cj{{\cal J}}
\def\zt{{\bold Z}_2}
\def\bs{{\bold s}}
\def\bof{{\bold f}}
\def\bq{{\bold Q}}
\def\be{{\bold e}}
\def\Hom{\on{Hom}}
\def\ker{\on{ker}}
\def\kot{\widetilde{KO}}
\def\coker{\on{coker}}
\def\da{\downarrow}
\def\colim{\operatornamewithlimits{colim}}
\def\zphat{\bz_2^\wedge}
\def\io{\iota}
\def\Om{\Omega}
\def\Prod{\prod}
\def\e{{\cal E}}
\def\zlt{\Z_{(2)}}
\def\exp{\on{exp}}
\def\abar{{\overline a}}
\def\xbar{{\overline x}}
\def\ybar{{\overline y}}
\def\zbar{{\overline z}}
\def\Rbar{{\overline R}}
\def\nbar{{\overline n}}
\def\cbar{{\overline c}}
\def\qbar{{\overline q}}
\def\bbar{{\overline b}}
\def\et{{\widetilde E}}
\def\ni{\noindent}
\def\coef{\on{coef}}
\def\den{\on{den}}
\def\lcm{\on{l.c.m.}}
\def\vi{v_1^{-1}}
\def\ot{\otimes}
\def\psibar{{\overline\psi}}
\def\thbar{{\overline\theta}}
\def\mhat{{\hat m}}
\def\exc{\on{exc}}
\def\ms{\medskip}
\def\ehat{{\hat e}}
\def\etao{{\eta_{\text{od}}}}
\def\etae{{\eta_{\text{ev}}}}
\def\dirlim{\operatornamewithlimits{dirlim}}
\def\gt{\widetilde{L}}
\def\lt{\widetilde{\lambda}}
\def\st{\widetilde{s}}
\def\ft{\widetilde{f}}
\def\sgd{\on{sgd}}
\def\lfl{\lfloor}
\def\rfl{\rfloor}
\def\ord{\on{ord}}
\def\gd{{\on{gd}}}
\def\rk{{{\on{rk}}_2}}
\def\nbar{{\overline{n}}}
\def\MC{\on{MC}}
\def\lg{{\on{lg}}}
\def\cB{\mathcal{B}}
\def\cS{\mathcal{S}}
\def\cP{\mathcal{P}}
\def\N{{\Bbb N}}
\def\Z{{\Bbb Z}}
\def\Q{{\Bbb Q}}
\def\R{{\Bbb R}}
\def\C{{\Bbb C}}
\def\l{\left}
\def\r{\right}
\def\mo{\on{mod}}
\def\xt{\times}
\def\notimm{\not\subseteq}
\def\Remark{\noindent{\it  Remark}}
\def\kut{\widetilde{KU}}

\def\*#1{\mathbf{#1}}
\def\0{$\*0$}
\def\1{$\*1$}
\def\22{$(\*2,\*2)$}
\def\33{$(\*3,\*3)$}
\def\ss{\smallskip}
\def\ssum{\sum\limits}
\def\dsum{\displaystyle\sum}
\def\la{\langle}
\def\ra{\rangle}
\def\on{\operatorname}
\def\od{\text{od}}
\def\ev{\text{ev}}
\def\o{\on{o}}
\def\U{\on{U}}
\def\lg{\on{lg}}
\def\a{\alpha}
\def\bz{{\Bbb Z}}
\def\eps{\varepsilon}
\def\bc{{\bold C}}
\def\bN{{\bold N}}
\def\nut{\widetilde{\nu}}
\def\tfrac{\textstyle\frac}
\def\b{\beta}
\def\G{\Gamma}
\def\g{\gamma}
\def\zt{{\Bbb Z}_2}
\def\pt{\widetilde{p}}
\def\zth{{\bold Z}_2^\wedge}
\def\bs{{\bold s}}
\def\bx{{\bold x}}
\def\bof{{\bold f}}
\def\bq{{\bold Q}}
\def\be{{\bold e}}
\def\lline{\rule{.6in}{.6pt}}
\def\xb{{\overline x}}
\def\xbar{{\overline x}}
\def\ybar{{\overline y}}
\def\zbar{{\overline z}}
\def\ebar{{\overline \be}}
\def\nbar{{\overline n}}
\def\rbar{{\overline r}}
\def\Mbar{{\overline M}}
\def\et{{\widetilde e}}
\def\ni{\noindent}
\def\ms{\medskip}
\def\ehat{{\hat e}}
\def\what{{\widehat w}}
\def\Yhat{{\widehat Y}}
\def\nbar{{\overline{n}}}
\def\minp{\min\nolimits'}
\def\mul{\on{mul}}
\def\N{{\Bbb N}}
\def\Z{{\Bbb Z}}
\def\Q{{\Bbb Q}}
\def\R{{\Bbb R}}
\def\C{{\Bbb C}}
\def\notint{\cancel\cap}
\def\se{\operatorname{secat}}
\def\cS{\mathcal S}
\def\cR{\mathcal R}
\def\el{\ell}
\def\TC{\on{TC}}
\def\dstyle{\displaystyle}
\def\ds{\dstyle}
\def\mt{\widetilde{\mu}}
\def\zcl{\on{zcl}}
\def\Vb#1{{\overline{V_{#1}}}}

\def\Remark{\noindent{\it  Remark}}
\title[Bounds for higher topological complexity]
{Bounds for higher topological complexity of real projective space implied by $BP$}
\author{Donald M. Davis}
\address{Department of Mathematics, Lehigh University\\Bethlehem, PA 18015, USA}
\email{dmd1@lehigh.edu}
\date{January 18, 2018}


\maketitle
\begin{abstract} We use Brown-Peterson cohomology to obtain lower bounds for the higher topological complexity, $\TC_k(RP^{2m})$, of real projective spaces, which are often much stronger than those implied by ordinary mod-2 cohomology.
 \end{abstract}
 \section{Introduction and main results}\label{introsec}
 In \cite{F}, Farber introduced the notion of topological complexity, $\TC(X)$, of a topological space $X$. This can be interpreted as one less than the minimal number of rules, called {\it motion planning rules}, required to tell how to move between any two points of $X$.\footnote{Farber's original definition did not include the ``one less than'' part, but most recent papers have defined it as we have done here.} This became central in the field of topological robotics when $X$ is the space of configurations of a robot or system of robots. This was generalized to higher topological complexity, $\TC_k(X)$, by Rudyak in \cite{R}. This can be thought of as one less than the number of rules required to tell how to move consecutively between any $k$ specified points of $X$ (\cite[Remark 3.2.7]{R}). In \cite{5}, the study of $\TC_k(P^n)$ was initiated, and this was continued in \cite{Dz}, where the best lower bounds implied by mod-2 cohomology were obtained. Here $P^n$ denotes real projective space.

 Since $\TC_2(P^n)$ is usually equal to the immersion dimension (\cite{FTY}), and a sweeping family of strong nonimmersion results was obtained using Brown-Peterson cohomology, $BP^*(-)$, in \cite{Annals}, one is led to apply $BP$ to obtain lower bounds for $\TC_k(P^n)$ for $k>2$. In this paper, we obtain a general result, Theorem \ref{genres}, which implies lower bounds in many cases, and then focus in Theorem \ref{specres} on a particular family of cases, which we show is often much stronger than the results implied by mod-2 cohomology.

 The general result is obtained from known information about the $BP$-cohomology algebra of products of real projective spaces. It gives conditions under which  nonzero classes of a certain form can be found. Here and throughout, $\nu(-)$ denotes the exponent of 2 in an integer.
 \begin{thm}\label{genres} Let $k\ge3$ and $r\ge0$. Suppose there are positive integers $a_1,\ldots,a_{k-1}$ whose sum is $km-(2^k-1)2^r$ such that
 \begin{equation}\label{hyp}\nu\biggl(\Prod_{i=1}^{k-1}\tbinom{a_i}{j_i}\biggr)\ge2^r\end{equation}
 for all $j_1,\ldots,j_{k-1}$ with $j_i\le m$ and $\ds\sum_{i=1}^{k-1}j_i\ge(k-1)m-(2^k-1)2^r$. Suppose also that
 \begin{equation}\label{prode}\nu\biggl(\sum_{\ell}\prod_{i=1}^{k-1}\tbinom{a_i}{m-\ell_i}\biggr)=2^r,\end{equation}
 where $\ell=(\ell_1,\ldots,\ell_{k-1})$ ranges over all $(k-1)$-tuples of the $k$ distinct numbers $2^{r+t}$, $0\le t\le k-1$. Then
 $$\TC_k(P^{2m})\ge 2km-(2^k-1)2^{r+1}.$$
 \end{thm}

 Theorem \ref{genres} applies in many cases, but we shall focus on one family. Here and throughout, $\a(-)$ denotes the number of 1's in the binary expansion of an integer.
 \begin{thm}\label{specres} Suppose $k\ge3$, $r\ge k-3$, and $m=A\cdot2^r$ with $A\ge 2^{k-1}$. Then
 $$\TC_k(P^{2m})\ge 2km-(2^k-1)2^{r+1}$$ if
 \begin{itemize}
 \item[a.] $k=3$ and either
 \begin{itemize}
 \item[i.] $A\equiv 5\ (8)$ and $\a(A)=2^r+2$, or
 \item[ii.] $A\equiv2\ (4)$ and $\a(A)=2^r+2$; or
 \end{itemize}
 \item[b.] $k\ge4$ and either
 \begin{itemize}
  \item[i.] $A\equiv6\ (8)$ and $\a(A)=2^r+2$, or
 \item[ii.] $A\equiv3\ (8)$ and $\a(A)=2^r+3$.
 \end{itemize}
 \end{itemize}
 \end{thm}

We prove Theorems \ref{genres} and \ref{specres} in Section \ref{pfsec}.
 In Section \ref{numsec}, we describe more specifically some families of particular values of $(m,k,r)$ to which this result applies, and the extent to which these results are much stronger than those implied by mod-2 cohomology. In Section \ref{sec4}, we prove that the cohomology-implied bounds for $\TC_k(P^n)$ are constant for long intervals of values of $n$. In these intervals, the $BP$-implied bounds become much stronger than those implied by cohomology.

 \section{Proofs of main theorems}\label{pfsec}
 In this section, we prove Theorems \ref{genres} and \ref{specres}. The first step, Theorem \ref{BPcohthm}, follows suggestions of Jesus Gonz\'alez, and is similar to work in \cite{5}. We are very grateful to Gonz\'alez for these suggestions. There are canonical elements $X_1,\ldots,X_k$ in $BP^2((P^n)^k)$, where $(P^n)^k$ is the
  Cartesian product of $k$ copies of $P^n$.
 \begin{thm}\label{BPcohthm} If $(X_1-X_k)^{a_1}\cdots(X_{k-1}-X_k)^{a_{k-1}}\ne0\in BP^*((P^n)^k)$, then
 $$TC_k(P^n)\ge2a_1+\cdots+2a_{k-1}.$$
 \end{thm}
 \begin{proof} Let $(P^n)^{[0,1]}$ denote the space of paths in $P^n$, and
$$P_{n,k}=(S^n)^k/((z_1,\ldots,z_k)\sim(-z_1,\ldots,-z_k))$$
a projective product space.(\cite{PPS}) The quotient map $P_{n,k}\mapright{\pi} (P^n)^k$  is a $(\zt)^{k-1}$-cover,  classified by a map $(P^n)^k\mapright{\mu}B((\zt)^{k-1})= (P^\infty)^{k-1}$. The map  $(P^n)^{[0,1]}\mapright{p} (P^n)^k$ defined by $$\sigma\mapsto(\si(0),\si(\tfrac1{k-1}),\ldots,\si(\tfrac{k-2}{k-1}),\si(1))$$ lifts to a map $(P^n)^{[0,1]}\mapright{\pt} P_{n,k}$.(\cite[(3.2)]{5}) A definition of $\TC_k(P^n)$ is as the sectional category $\se(p)$. The lifting $\pt$ implies that $\se(p)\ge\se(\pi)$.

Let $G=(\zt)^{k-1}$, and $B_tG=(\ds\Asterisk^{t+1}G)/G$, where $\ds\Asterisk^{t+1}G$ denotes the  iterated join of $t+1$ copies of $G$. Note that $B_tG$ is the $t$th  stage in Milnor's construction of  $BG$, with a map $i_t:B_tG\to BG$. By \cite[Thm 9, p.~86]{Sw}, as described in \cite[(4.1)]{5}, $\mu$ lifts to a map $(P^n)^k\mapright{\mt}B_{\se(\pi)}G$.

\begin{tikzcd}
(P^n)^{[0,1]} \ar{r}{\pt} \ar{rd}{p}&P_{n,k}\ar{d}{\pi}&B_{\se(\pi)}G\ar{d}{i_{\se(\pi)}}\\
&(P^n)^k \ar{r}{\mu} \ar{ur}{\mt}&BG=(P^\infty)^{k-1}
\end{tikzcd}

By \cite[Prop 3.1]{5}, $\mu$ classifies $(p_1^*(\xi)\ot p_k^*(\xi))\times\cdots\times (p_{k-1}^*(\xi)\ot p_k^*(\xi))$, and so, by \cite[Prop 3.6]{As}, the induced homomorphism
$$BP^*((P^\infty)^{k-1})\mapright{\mu^*} BP^*((P^n)^k)$$
satisfies $\mu^*(X_i)=u_i(X_i-X_k)$ for $1\le i\le k-1$, with $u_i$ a unit.  Since $\mu^*=\mt^*i_{\se(\pi)}^*$ and $B_tG$  is $t$-dimensional, $\mu^*(X_1^{a_1}\cdots X_{k-1}^{a_{k-1}})=0$ if $2a_1+\cdots+2a_{k-1}>\se(\pi)$. The theorem now follows since $\prod(X_i-X_k)^{a_i}\ne0$ implies $\mu^*(\prod X_i^{a_i})\ne0$, which implies
$$\sum 2a_i\le\se(\pi)\le\se(p)=\TC_k(P^n).$$\end{proof}

We use this to prove Theorem \ref{genres}.
\begin{proof}[Proof of Theorem \ref{genres}]
Let $I$ denote the ideal $(v_0,\ldots,v_k)\subset BP^*$. Recall $v_0=2$ and $|v_i|=2(2^i-1)$. In $BP^*(X)$, let $F_s$ denote the $BP^*$-submodule $I^s\cdot BP^*(X)$.  It follows from
\cite[2.2]{SW}, \cite[Cor 2.4]{D2}, and \cite[Thm 1.10]{Ds} that in  $BP^*((P^{2m})^k)$, for $r\ge0$ and integers $j_1,\ldots,j_k$,
\begin{equation}\label{DR}2^{2^r}X_1^{j_1}\cdots X_k^{j_k}\equiv v_k^{2^r}\sum X_1^{j_1+\ell_1}\cdots X_k^{j_k+\ell_k}\mod F_{2^r+1},\end{equation}
where the sum is taken over all permutations $(\ell_1,\ldots,\ell_k)$ of $\{2^r,\ldots,2^{r+k-1}\}$. (An analogous result was  derived in $BP$-homology in \cite{Ds}, following similar, but not quite so  complete, results in \cite{SW} and \cite{D2}, which also discussed the dualization to obtain $BP$-cohomology results.)

 The result follows from Theorem \ref{BPcohthm} once we show that
$$(X_1-X_k)^{a_1}\cdots(X_{k-1}-X_k)^{a_{k-1}}\ne0\in BP^{2km-(2^k-1)2^{r+1}}((P^{2m})^k).$$
This expands as $\ds\sum_{j_1,\ldots,j_{k-1}}\pm\tbinom {a_1}{j_1}\cdots\tbinom {a_{k-1}}{j_{k-1}}X_1^{j_1}\cdots X_{k-1}^{j_{k-1}}X_k^{km-(2^k-1)2^r-j_1-\cdots-j_{k-1}}$, for values of $j_1,\ldots,j_{k-1}$ described in Theorem \ref{genres}. By (\ref{DR}) and (\ref{hyp}), this equals, mod $F_{2^r+1}$,
\begin{equation}\label{sumsum}v_k^{2^r}\sum_{j_1,\ldots,j_{k-1}}\sum_\ell\pm2^{-2^r}\tbinom {a_1}{j_1}\ldots\tbinom {a_{k-1}}{j_{k-1}}X_1^{j_1+\ell_1}\cdots X_{k-1}^{j_{k-1}+\ell_{k-1}}X_k^{km-j_1-\ell_1-\cdots-j_{k-1}-\ell_{k-1}},\end{equation}
with $\ell=(\ell_1,\ldots,\ell_{k-1})$ as in (\ref{prode}). Note here that $\ell_k=2^{r+k}-2^r-\ell_1-\cdots-\ell_{k-1}$.
The terms in (\ref{sumsum}) are 0 unless the exponent of each $X_i$ equals $m$, since otherwise there would be a factor $X^p$ with $p>m$. We are left with
$$\bigl(\sum_{\ell}\pm2^{-2^r}\tbinom {a_1}{m-\ell_1}\cdots\tbinom {a_{k-1}}{m-\ell_{k-1}}\bigr)v_k^{2^r}X_1^m\cdots X_k^m$$
with $(\ell_1,\ldots,\ell_{k-1})$  as above, and this is nonzero by the hypothesis (\ref{prode}) and the fact, as was noted in \cite{SW}, that by the (proven) Conner-Floyd conjecture,
$v_k^hX_1^m\cdots X_k^m\ne0$ for any nonnegative integer $h$.\end{proof}

In the following proof of Theorem \ref{specres}, we will often use without comment Lucas's Theorem regarding binomial coefficients mod 2, and that
\begin{equation}\label{methods}\nu\tbinom mn=\a(n)+\a(m-n)-\a(m),\text{ and }\a(x-1)=\a(x)-1+\nu(x).\end{equation}
\begin{proof}[Proof of Theorem \ref{specres}] We explain the proof when $k\ge4$ and $A\equiv6\ (8)$, and then describe the minor changes required when $A\equiv3$ or $k=3$. We apply Theorem \ref{genres} with $$a_i=m-(2^k-1)2^{r-i},\ 1\le i\le k-3,\ a_{k-2}=m,\text{ and }a_{k-1}=2m-(2^k-1)2^{r-(k-3)}.$$

For (\ref{hyp}), we show $$\nu\tbinom{a_{k-1}}j\ge2^r\text{ if }(k-1)m-(2^k-1)2^r-(a_1+\cdots+a_{k-2})\le j\le m.$$ Thus we are considering $\nu\binom{2m-(2^k-1)2^{r-(k-3)}}j$ with $m-(2^{k}-1)2^{r-(k-3)}\le j\le m$. By symmetry, we may restrict to $m-(2^{k-1}-1)2^{r-(k-3)}\le j\le m$. Let $m=(8B+6)2^r$ with $\a(B)=2^r$. We first restrict to $j$'s divisible by $2^{r-(k-3)}$; let $j=2^{r-(k-3)}h$. Now we are considering $\nu\binom{(8B+6)2^{k-2}-2^k+1}h$ with $2^{k-3}(8B+6)-(2^{k-1}-1)\le h\le 2^{k-3}(8B+6)$.
Lemma \ref{techlem} with $t=k-2$ shows that $\nu\binom{(8B+6)2^{k-2}-2^k+1}h\ge\a(B)$ for the required values of $h$. The proof for arbitrary $j$ (in the required range) follows from the easily proved fact that
$$\text{for }0<\delta<2^k,\quad \nu\tbinom{N\cdot2^k}{M\cdot2^k+\delta}>\nu\tbinom{N\cdot2^k}{M\cdot2^k}.$$

Now we prove (\ref{prode}). We divide the top and bottom of the binomial coefficients by $2^{r-(k-3)}$; this does not change the exponent. The tops are now
$$2^{k-3}A-(2^k-1)2^{k-4},\ldots,2^{k-3}A-(2^k-1)2^0,\ 2^{k-3}A,\ 2^{k-2}A-(2^k-1),$$
and the bottoms are selected from $2^{k-3}A-2^{k-3},\ldots,2^{k-3}A-2^{2k-4}$. All the bottoms except the last one are greater than the first top one.
Thus to get a nonzero product in (\ref{prode}), the last bottom must accompany the first top, and after dividing top and bottom by $2^{k-4}$, it becomes $\binom{2A-(2^k-1)}{2A-2^k}\equiv 1$ mod 2.
Similar considerations work inductively for all but the final two factors, showing that the $i$th bottom from the end must appear beneath the $i$th top and gives an odd factor.
What remains is $$\sum\tbinom{2^{k-3}A}{j}\tbinom{2^{k-2}A-2^k+1}{j'},$$
where $(j,j')$ are the ordered pairs of distinct elements of $$\{2^{k-3}A-2^{k-3},2^{k-3}A-2^{k-2},2^{k-3}A-2^{k-1}\}.$$ The $+1$ on top does not affect the exponent of the binomial coefficients, and so we may remove it and then divide tops and bottoms by $2^{k-3}$, obtaining $\sum\binom Aj\binom{2A-8}{j'}$, where $(j,j')$ are ordered pairs of $A-1$, $A-2$, and $A-4$.

If $A\equiv6$ mod 8, $\nu\binom Aj=0$ if $j=A-2$ or $A-4$, and is $>0$ if $j=A-1$. Also, with $A=8B+6$, $\nu\binom{2A-8}{j'}=\a(B)$ if $j'=A-2$, and is $>\a(B)$ if $j'=A-1$ or $A-4$. Thus the sum in (\ref{prode}) has $\nu(-)=2^r$, coming from the single summand corresponding to $(j,j')=(A-4,A-2)$.

When $A\equiv3$ mod 8, the following minor changes must be made in the above argument. Let $A=8B+3$. A minimal value of $\nu\binom{a_{k-1}}j$ occurs when $j=2^{r-(k-3)}h$ with $h=2^{k-3}(8B+3)-2^{k-3}$. We obtain $\nu\binom{16B-2}{8B+2}=\a(B)-1=2^r$ since $\a(A)=2^r+3$. For (\ref{prode}), the minimal value  $\nu\bigl(\binom Aj\binom{2A-8}{j'}\bigr)=2^r$ occurs only for $(j,j')=(A-2,A-1)$.

Part (a) of Theorem \ref{specres} follows similarly. We have $a_1=m$ and $a_2=2m-7\cdot2^r$. Then by the same methods as used above, we show that with $m$ as in the theorem, and $P$ denoting a positive number
and $I$ a number which is irrelevant,
\begin{itemize}
\item[$\bullet$] If $m-7\cdot2^r\le j\le m$, then $\nu\binom{2m-7\cdot2^r}j\ge2^r$.
\item[$\bullet$] The values $(\nu\binom m{m-2^r},\nu\binom m{m-2^{r+1}},\nu\binom m{m-2^{r+2}})$ are $(0,P,0)$ (resp. $(P,0,I)$) in case (i) (resp. (ii)) of the theorem.
\item[$\bullet$] The values $(\nu\binom{2m-7\cdot2^r}{m-2^r}-2^r,\nu\binom{2m-7\cdot2^r}{m-2^{r+1}}-2^r,\nu\binom{2m-7\cdot2^r}{m-2^{r+2}})-2^r)$
are $(P,0,0)$ (resp. $(0,0,P)$) in case (i) (resp. (ii)) of the theorem.
\end{itemize}
\end{proof}

The following lemma was used above.
\begin{lem}\label{techlem} If $t\ge2$ and $-2^t+1\le d\le 2^t$, then $\nu\binom{(8B+2)2^t+1}{(4B+2)2^t+d}\ge\a(B)$.\end{lem}
\begin{proof} Using (\ref{methods}), we can show
$$\nu\binom{(8B+2)2^t+1}{(4B+2)2^t+d}=\begin{cases}\a(B)+t+1-\nu(d(d-1))&-2^t+1\le d<0\\
\a(B)&d=0,1\\
\a(B)+t+\nu(B)+2-\nu(d(d-1))&2\le d\le 2^t,\end{cases}$$
from which the lemma is immediate.\end{proof}

\section{Numerical results}\label{numsec}
In this section, we compare the lower bounds for $\TC_k(P^{2m})$ implied by $BP$ with those implied by mod-2 cohomology. In \cite{Dz}, the best lower bounds obtainable using mod-2 cohomology were obtained. They are restated here in (\ref{zres}). In Table \ref{T2}, we compare these with the results implied by our Theorems \ref{genres} and \ref{specres} for $\TC_3(P^{2m})$ with $32\le m<63$. Results in the $BP$ column are those implied by \ref{genres}, and those indicated with an asterisk are implied by \ref{specres}. It is quite possible that there are additional results implied by Theorem \ref{BPcohthm},
since Theorem \ref{genres} takes into account only one type of implication about nonzero classes in $BP^*((P^n)^k)$. Note that the $BP$-bounds are significantly stronger in the second half of the table.

\begin{table}[H]
\caption{Lower bounds for $\TC_3(P^{2m})$ implied by $H^*(-)$ and by $BP$}
\label{T2}
\begin{tabular}{c|cl}
$m$&$H^*(-)$&$BP$\\
\hline
$32$&$192$&$152$\\
$33$&$198$&$152$\\
$34$&$204$&$190$\\
$35$&$206$&$190$\\
$36$&$216$&$190$\\
$37$&$222$&$208*$\\
$38$&$222$&$214*$\\
$39$&$222$&$214*$\\
$40$&$240$&$214*$\\
$41$&$246$&$232$\\
$42$&$252$&$238*$\\
$43$&$254$&$238*$\\
$44$&$254$&$238*$\\
$45$&$254$&$238*$\\
$46$&$254$&$248$\\
$47$&$254$&$248$\\
$48$&$254$&$248$\\
$49$&$254$&$280$\\
$50$&$254$&$286*$\\
$51$&$254$&$286*$\\
$52$&$254$&$286*$\\
$53$&$254$&$304$\\
$54$&$254$&$310$\\
$55$&$254$&$310$\\
$56$&$254$&$310$\\
$57$&$254$&$310$\\
$58$&$254$&$320*$\\
$59$&$254$&$320*$\\
$60$&$254$&$332*$\\
$61$&$254$&$332*$\\
$62$&$254$&$332*$\\
$63$&$254$&$332*$
\end{tabular}
\end{table}

In Table \ref{T3}, we present another comparison of the results implied by Theorem \ref{specres} and those implied by ordinary mod-2 cohomology. We consider lower bounds for $\TC_4(P^{2m})$ for $2^{11}\le m<2^{12}$. In Table \ref{T3}, the first column refers to a range of values of $m$, the second column to the number of distinct new results implied by Theorem \ref{specres} in that range, and the third column to the range of the ratio of  bounds implied by Theorem \ref{specres} to those implied by ordinary cohomology.
There are many other stronger bounds implied by $BP$ via Theorem \ref{genres}, but our focus here is on the one family which we have analyzed for all $k$ and $r$.

\begin{table}[H]
\caption{Ratio of lower bounds for $\TC_4(P^{2m})$ implied by Theorem \ref{specres} to those implied by $H^*(-)$}
\label{T3}
\begin{tabular}{c|cc}
$m$&\#&ratio\\
\hline
$[2048,2815]$&$29$&$[.9620,1.0384]$\\
$[2816,3071]$&$7$&$[.9877,1.0673]$\\
$[3072,3979]$&$26$&$[.9783,1.2700]$\\
$[3980,4095]$&$1$&$1.2908$
\end{tabular}
\end{table}

In the range $2816\le m\le 3071$ here, the bound for $\TC_4(P^{2m})$ implied by mod-2 cohomology is constant at 22525, while that implied by Theorem \ref{specres} increases from 22248 to 24040.
In the longer range $3072\le m\le4095$ here, the bound for $\TC_4(P^{2m})$ implied by mod-2 cohomology is constant at 24573, while that implied by Theorem \ref{specres} increases from 24040 to 31720. Next, we examine what happens in the generalization of this latter range to $\TC_k(P^{2m})$ for arbitrary $k$ and arbitrary 2-power near the end of the range. In Theorem \ref{zthm}, we will show that the bound for $\TC_k(P^{2m})$ implied by cohomology has the constant value $(k-1)(2^e-1)$ for $[\frac{k-1}k\cdot2^e]\le 2m\le2^e-1$.

In this range, the  bound implied by Theorem \ref{specres} will increase from a value approximately equal to the cohomology-implied bound to a value which, as we shall explain, is asymptotically as much greater than the cohomology-implied bound as it could possibly be.
The following result gives a result at the end of each 2-power interval, since each $e$ can be written uniquely as $2^r+r+3+d$ for $0\le d\le 2^r$. For example, the case $r=1$, $d=0$, $k=3$ in this proposition is the $332*$ next to $m=60$ in Table \ref{T2}, and the case $r=2$, $d=3$, $k=4$ gives $m=3980$, the start of the last row of Table \ref{T3}.
\begin{prop}\label{prop} For $r\ge1$ and $0\le d\le 2^r$, let
$$m=\begin{cases}2^{r+1}(2^{2^r+2}-1)&d=0,\ k\ge3\\
2^{r+d+2}(2^{2^r+1}-1)+2^{r+1}&d>0,\ k=3\\
2^{r+d+2}(2^{2^r+1}-1)+3\cdot2^{r}&d>0,\ k>3.\end{cases}$$
Then $\TC_k(P^{2m})\ge 2km-(2^k-1)2^{r+1}$.\end{prop}
\begin{proof} It is straightforward to check that the conditions of Theorem \ref{specres} are satisfied for these values of $m$ and $r$.\end{proof}

For $m$ as in Proposition \ref{prop}, the lower bound for $\TC_k(P^{2m})$ implied by cohomology is $(k-1)(2^{2^r+r+4+d}-1)$. One can check that the ratio of the bound in Proposition \ref{prop} to the cohomology bound is greater than
$$\frac k{k-1}-\frac 1{2^{2^r+1}}.$$
Since, as was noted in \cite{5}, $(k-1)n\le \TC_k(P^n)\le kn$, the largest the ratio of any two estimates of $\TC_k(P^n)$ could possibly be is $k/(k-1)$. Thus the $BP$-bound improves on the cohomology bound asymptotically by as much as it possibly could, as $e$ (hence $r$) becomes large. 

Jesus Gonz\'alez (\cite{5}) has particular interest in estimates for $\TC_k(P^{3\cdot2^e})$. We shall prove the interesting fact that our Theorems \ref{genres} and \ref{specres} improve significantly on the cohomological lower bound for $\TC_3(P^{3\cdot2^e})$, but not for $\TC_k(P^{3\cdot2^e})$ when $k>3$.

The bound implied by cohomology (Theorem \ref{zthm}) is \begin{equation}\label{z3}\TC_k(P^{3\cdot2^e})\ge(k-1)(2^{e+2}-1).\end{equation}
Since $2km-(2^k-1)2^{r+1}\le(k-1)(2^{e+2}-1)$ if $k\ge4$ and $m\le 3\cdot2^e$ (and $r\ge0$), Theorem \ref{genres} cannot possibly improve on (\ref{z3}) if $k\ge4$. In order for $BP$ to possibly improve on (\ref{z3}) when $k\ge4$, a much more delicate analysis of $BP^*((P^n)^k)$ would have to be performed, involving new ways of showing that classes are nonzero, and then using Theorem \ref{BPcohthm}.

However, Theorem \ref{specres} implies a lower bound for $\TC_3(P^{3\cdot2^e})$ which is asymptotically 9/8 times the bound in (\ref{z3}).
\begin{thm}\label{3thm} Let $r\ge1$, $0\le d\le 2^r$, and $e=2^r+r+d+3$. Then
$$\TC_3(P^{3\cdot2^e})\ge9\cdot2^e-3\cdot2^{r+3+d}-2^{r+1}.$$
\end{thm}
\begin{proof} One easily checks that, with $e$ as in the theorem, $m=3\cdot2^{e-1}-2^{r+2+d}+2^{r+1}$ satisfies the hypothesis of Theorem \ref{specres}(a)(ii), and that Theorem \ref{specres} then implies $\TC_3(P^{2m})\ge 9\cdot2^e-3\cdot2^{r+3+d}-2^{r+1}$, implying this theorem by naturality.\end{proof}

In Table \ref{T5}, we compare the bounds for $\TC_3(P^{3\cdot2^e})$ implied by Theorem \ref{3thm} and by (\ref{z3}) for various values of $e$. Every $e$ has a unique $r$ and $d$. The $m$-column is the value of $m<3\cdot2^{e-1}$ which appears in the proof of \ref{3thm}. The ``$BP$-bound'' column is the bound for $\TC_3(P^{3\cdot2^e})$ given by Theorem \ref{3thm}, and the ``$H^*$-bound'' column that is given by (\ref{z3}). The final column is the ratio of the $BP$-bound to the $H^*$-bound, which approaches 1.125 as $e$ gets large.

\begin{table}[H]
\caption{Ratio of lower bounds for $\TC_3(P^{3\cdot2^e})$ implied by Theorem \ref{3thm} to those implied by $H^*(-)$}
\label{T5}
\begin{tabular}{c|cccccl}
$e$&$r$&$d$&$m$&$BP$-bound&$H^*$-bound&ratio\\
\hline
$6$&$1$&$0$&$92$&$524$&$510$&$1.027$\\
$7$&$1$&$1$&$180$&$1052$&$1022$&$1.029$\\
$8$&$1$&$2$&$356$&$2108$&$2046$&$1.030$\\
$9$&$2$&$0$&$760$&$4504$&$4094$&$1.100$\\
$10$&$2$&$1$&$1512$&$9016$&$8190$&$1.101$\\
$11$&$2$&$2$&$3016$&$18040$&$16382$&$1.101$\\
$22$&$3$&$8$&&&&$1.1235$\\
$23$&$4$&$0$&&&&$1.124994$
\end{tabular}
\end{table}

Using different choices of $a_1$ and $a_2$ (found by computer), Theorem \ref{genres} can do somewhat better for $\TC_3(P^{3\cdot2^e})$ than Theorem \ref{specres}, but it does not seem worthwhile to try to find the best result implied by Theorem \ref{genres} for all $e$, since no pattern is apparent. For $e$ from 7 to 11, the lower bounds for $\TC_3(P^{3\cdot2^e})$ implied by Theorem \ref{genres} are, respectively, 1072, 2224, 4516, 9068, and 18284. For example, when $e=11$, it is about 1.4\% better than that implied by Theorem \ref{3thm} and 11.6\% better than that implied by cohomology.
For one who wishes to check this result when $e=11$, use $m=3066$, $r=3$, and $a_1=3287$ in Theorem \ref{genres}. The values of $\nu\binom{a_1}{m-2^{r+\eps}}$ (resp.~$\nu\binom{a_2}{m-2^{r+\eps}}$) for $\eps=0,1,2$ are (5,6,7) (resp.~(6,6,3)).
 \section{$\TC_k(P^n)$  result implied by mod-2 cohomology, in a range}\label{sec4}
 In this section, we prove that the lower bound for $\TC_k(P^n)$ implied by cohomology is constant in the last $\frac2k$ portion of the interval between successive 2-powers.
 This generalizes the behavior seen  in Table \ref{T2} ($k=3$) or Table \ref{T3} ($k=4$) .
 In the previous section, we showed that the bound implied by $BP$ rises in this range to a value nearly $k/(k-1)$ times that of the cohomology bound, which is as much as it possibly could.

 Recall from \cite{5} or \cite{Dz} that $\zcl_k(P^n)$ is the lower bound for $\TC_k(P^n)$ implied by mod-2 cohomology. It is an analogue of Theorem \ref{BPcohthm}, except that classes are in grading 1 rather than grading 2. Here we prove the following new result about $\zcl_k(P^n)$.
\begin{thm}\label{zthm} For $k\ge3$ and $e\ge2$,  $\zcl_k(P^n)=(k-1)(2^e-1)$ for $[\frac{k-1}k\cdot2^e]\le n\le2^e-1$.\end{thm}

Note that, since $(k-1)n\le\zcl_k(P^n)\le kn$ (by \cite{5} or \cite{Dz}), this interval of constant $\zcl_k(P^n)$ is as long as it could possibly be.

\begin{proof} We rely on \cite[Thm 1.2]{Dz}, which can be interpreted to say that, with $n_t$ denoting $n$ mod $2^t$,
\begin{equation}\label{zres}\zcl_k(P^n)=kn-\max(2^{\nu(n+1)}-1,kn_t-(k-1)(2^t-1)),\end{equation}
with the max taken over all $t$ for which the initial bits of $n$ mod $2^t$ begin a string of at least two consecutive 1's.
That $\zcl_k(P^{2^e-1})=(k-1)(2^e-1)$ is immediate from (\ref{zres}). Since $\zcl_k(P^n)$ is an increasing function of $n$, it suffices to prove
\begin{equation}\label{init}\text{if }n=[\tfrac{k-1}k\cdot2^e],\text{ then }\zcl_k(P^n)=(k-1)(2^e-1).\end{equation}

The case $k=3$ is slightly special since the binary expansion of $n=[2^{e+1}/3]$ does not have any consecutive 1's. For this $n$,  (\ref{zres}) implies that $\zcl_3(P^n)=3n+1-2^{\nu(n+1)}=2^{e+1}-2$, as desired. From now on, we assume $k>3$ in this proof.

One part that we must prove is
\begin{equation}\label{2pow}kn-2^{\nu(n+1)}+1\ge(k-1)(2^e-1)\end{equation}
if $n$ is as in (\ref{init}). Write $2^e=Ak-\delta$ with $0\le\delta\le k-1$. Then $n=2^e-A$, and the desired inequality reduces to $k-\delta\ge2^{\nu(A-1)}$ since $\nu(A-1)=\nu(2^e-A+1)$. If $A-1=2^tu$ with $u$ odd, then $k-\delta=2^e-2^tuk\ge 2^t$ since $k-\delta>0$, proving the inequality.

The rest of the proof requires the following lemma.
\begin{lem}\label{binexplem} Let $k$ be odd, and $\be$ the multiplicative order of 2 mod $k$. Thus $\be$ is the smallest positive integer such that $k$ divides $2^\be-1$. Let $m=(k-1)\frac{2^\be-1}k$, and let $B$ be the binary expansion of $m$. If $t=\a\be+\b$ with $0\le \b<\be$, then the binary expansion of $[(k-1)2^t/k]$ consists of the concatenation of $\a$ copies of $B$, followed by the first $\b$ bits of $B$. Also, the binary expansion of $[(2^vk-1)2^{v+t}/(2^vk)]$ with $k$ odd equals that of $[(k-1)2^t/k]$ preceded by $v$ 1's. If $k\ge4$, $B$ begins with at least two 1's.\end{lem}
\begin{proof} Let $f_t=(k-1)2^t/k$. Then, letting $\{f\}=f-[f]$ denote the fractional part of $f$,
$$[f_{t+1}]=\begin{cases}2[f_t]&\text{if }\{f_t\}<1/2\\ 2[f_t]+1&\text{ if }\{f_t\}\ge1/2.\end{cases}$$
This shows that as $t$ increases, the binary expansions of the $[f_t]$ are just initial sections of subsequent ones. They start with at least two 1's when $k\ge4$ since $[2^2(k-1)/k]=3$.

If $\be$ is as in the lemma, then
$$\frac{(k-1)2^{t+\be}}k-\frac{(k-1)2^t}k=2^t\frac{(k-1)(2^\be-1)}k,$$
showing that adding this $\be$ to the exponent just appends $B$ in front of the binary expansion. Regarding $2^vk$, note that
$$\frac{(2^vk-1)2^{t+v}}{2^vk}=(2^v-1)2^t+\frac{(k-1)2^t}k,$$
which shows the appending of 1's in front.\end{proof}

In Table \ref{T4}, we list some values of $B$, the binary expansion of $m$, for the $m$ associated to $k$ as in Lemma \ref{binexplem}.

\begin{table}[h]
\caption{Binary expansions $B$ of numbers appearing in lemma}
\label{T4}
\begin{tabular}{c|cl}
$k$&$\be$&$B$\\
\hline
$9$&$6$&$111000$\\
$11$&$10$&$1110100010$\\
$13$&$12$&$111011000100$\\
$15$&$4$&$1110$\\
$17$&$8$&$11110000$\\
$19$&$18$&$111100101000011010$\\
$21$&$6$&$111100$\\
$23$&$11$&$11110100110$
\end{tabular}
\end{table}

The property (\ref{bin}) says roughly that the beginning of $B$ has more 1's than anywhere else in $B$.

For any $k>3$ and $n=[\frac{k-1}k\cdot2^e]$ as in (\ref{init}), equations (\ref{zres}) and (\ref{2pow}) imply that
$$\zcl_k(P^n)\le kn-(kn-(k-1)(2^e-1))=(k-1)(2^e-1),$$
with equality if, for all $t$ for which the initial bits of $n$ mod $2^t$ begin a string of at least two consecutive 1's,
$$kn_t-(k-1)(2^t-1)\le kn-(k-1)(2^e-1).$$
This is equivalent to
\begin{equation}\label{comp}1-\tfrac1k\le \frac{n-n_t}{2^e-2^t}.\end{equation}

By the lemma, if $k$ is odd (resp.~even), the RHS of (\ref{comp}) is the same as (resp.~greater than) it would be if $(n,e)$ is replaced by $(m,\be)$, with notation as in the lemma, provided $t\le \be$.
Note that equality holds in (\ref{comp}) if $(n,e,t)$ is replaced by $(m,\be,0)$. Hence, again using the lemma for cases in which $t>\be$, (\ref{comp}) will follow from its validity if $(n,e)$ is replaced by $(m,\be)$, and, since $1-\frac1k=\frac{m}{2^\be-1}$, this reduces to showing \begin{equation}\label{bin}\tfrac{m_t}{2^t-1}\le \tfrac m{2^\be-1}.\end{equation}

Let $q=\frac{2^\be-1}k=2^\be-1-m$ and $q_t=2^t-1-m_t$ its reduction mod $2^t$. Now the desired inequality reduces to $\frac{q_t}{2^t-1}\ge\frac q{2^\be-1}=\frac1k$; i.e., $k q_t\ge2^t-1$. We can prove the validity of this last inequality as follows. Write $q=q_t+2^t\a$, for an integer $\a$. Then
$$2^\be-1=kq=kq_t+2^t\a k.$$
Reducing mod $2^t$ gives the desired result.
\end{proof}
\begin{rmk} {\rm It appears that the stronger inequality $k q_t\ge 3\cdot2^t-1$ holds when $q=\frac{2^\be-1}k$, but we do not need it, and it seems much harder to prove.}\end{rmk}

 \def\line{\rule{.6in}{.6pt}}

\end{document}